\documentclass[reqno]{amsart}
\usepackage{amssymb}
\usepackage{hyperref}
\newtheorem{theorem}{Theorem}[section]

\newtheorem{prop}[theorem]{Proposition}
\newtheorem{cor}[theorem]{Corollary}
\newtheorem{thm}[theorem]{Theorem}
\newtheorem{defn}[theorem]{Definition}
\newtheorem{example}[theorem]{Example}

\newtheorem{rem}[theorem]{Remark}

\numberwithin{equation}{section}
\title[The inverse of the confluent Vandermonde matrices]{New approaches for solving linear confluent Vandermonde systems and inverse of their corresponding matrices via Taylor's expansion}
\author{ Mohammed Mou\c{c}ouf and Said Zriaa}
\date{}
\begin{document}
\maketitle
\begin{center}
{\footnotesize Department of Mathematics, Faculty of Science, Chouaib Doukkali University, Morocco\\
Email: moucouf@hotmail.com\\
Email: saidzriaa1992@gmail.com}
\end{center}

\begin{abstract} In this paper, we present a novel method to compute an explicit formula for the inverse of the confluent Vandermonde matrices. Our proposed results may have many interesting perspectives in diverse areas of mathematics and natural sciences, notably on the situations where the  Vandermonde matrices have acquired much usefulness. The method presented here is direct and straightforward, it gives explicit and compact formulas not exist in the large literature on this essential topic. Several examples are presented to highlight the results. Our main tools are some elementary basic linear algebra.
\end{abstract}

\section{Introduction}
The Vandermonde matrices are of considerable importance in diverse areas of mathematics, physics, and engineering. Both the Vandermonde matrices and its inverses have been widely used in many applications include interpolation, discrete Fourier transform, coding theory, hypersurfaces, systems theory, etc. The necessity of such a result is needed and required. It is a well-known fact that the direct calculation of the inverse of such matrices is a very difficult problem numerically and algebraically. Methods to compute the inverse of a Vandermonde matrix are very well studied. D. Kalman shows in \cite{Dk} that Vandermonde matrices have realized serious importance in differential and difference equations due to their normal appearance. The computation of the inverse of the usual Vandermonde matrix is relatively feasible and is provided in the literature but the inverse of the confluent Vandermonde matrix is completely difficult to achieve. Therefore, many theoretical and numerical methods have been developed for the computation of the inverse of the usual Vandermonde matrix given in the literature by many authors for example~\cite{Ali,Meael,Apant,Eara,As2,Swor}. By remarking that there is a solid relation between interpolation polynomials and inversion of the confluent Vandermonde matrices, we can derive an efficient and simple method to inverse this types of matrices. An inversion procedure for such a matrix has been produced in \cite{Magraw} only in a particular case where the multiplicities are equal to $2$. Our general method is entirely different.\\
\indent In the literature, several methods and algorithms have been presented for calculating the inverse of the confluent Vandermonde matrices (see for example \cite{Ho1,Ho2,Jsres,Tang}). However, the most widely used method is the partial fraction expansion it aims to give a recursive algorithm for computing the inverse of the confluent Vandermonde matrices but the result is not presented in a more compact form (see for example \cite{Ho1}).\\
\indent Using the change of basis matrix, we exhibit the general solution to this problem, as an application we derive simply the inverse of another type of Vandermonde matrices. Our results are obtained simply using minimum knowledge. Detailed examples and concluding remarks are presented.\\
\section{Explicit inverse of confluent Vandermonde matrices and applications to solving corresponding linear systems}
 In this section, we present our method to obtain the explicit inverse of the confluent Vandermonde matrices.\\
 Let
\begin{equation}\label{eq: 2.1}
P(x)=(x-\alpha_{1})^{m_{1}}(x-\alpha_{2})^{m_{2}}\cdots(x-\alpha_{s})^{m_{s}}
\end{equation}
be a polynomial of degree $n$ and $\alpha_{1},\alpha_{2},\ldots,\alpha_{s}$ be distinct elements of a field $\mathbb{F}$ of characteristic zero, and $m_{1},m_{2},\ldots,m_{s}$ positive integers. Let
\begin{equation}\label{eq: 2.2}
L_{jk_{j}}(x)[P]=P_{j}(x)(x-\alpha_{j})^{k_{j}}\sum_{i=0}^{m_{j}-1-k_{j}}\frac{1}{i!}g^{(i)}_{j}(\alpha_{j})(x-\alpha_{j})^{i}
\end{equation}
where $1\leq j \leq s;0\leq k_{j} \leq m_{j}-1$ and
\begin{equation}
P_{j}(x)=\prod_{i=1,i\neq j}^{s}(x-\alpha_{i})^{m_{i}}=\frac{P(x)}{(x-\alpha_{j})^{m_{j}}},1\leq j \leq s
\end{equation}
and
\begin{equation}\label{eq: 2.4}
g_{j}(x)=(P_{j}(x))^{-1}
\end{equation}
The polynomials $L_{jk_{j}}(x)[P]$ satisfy $L^{(l)}_{jk_{j}}(\alpha_{i})[P]=l!\delta_{ij}\delta_{lk_{j}}$, here and further $L^{(l)}_{jk_{j}}(x)[P]$ means the $l$th derivative of $L_{jk_{j}}(x)[P]$ (for more details, see \cite{As1}).\\
 The generalization of Hermite's interpolation formula allowed us to write every polynomial in terms of its derivatives at any given point. Exactly in his paper, A. Spitzbart \cite{As1} generalize Hermite's interpolation formula which permits us to write every polynomial $Q$ of degree less than or equal to $n-1$ as
\begin{equation}\label{eq: 2.5}
Q=\sum_{j=1}^{s}(\sum_{k_{j}=0}^{m_{j}-1}\frac{1}{k_{j}!}Q^{(k_{j})}(\alpha_{j})L_{jk_{j}}(x)[P])
\end{equation}
These formulas will be used later to get the researched inverse of the confluent Vandermonde matrix.\\
\indent D. Kalman \cite{Dk} defined the confluent Vandermonde matrix associated with the polynomial $P(x)=(x-\alpha_{1})^{m_{1}}(x-\alpha_{2})^{m_{2}}\cdots(x-\alpha_{s})^{m_{s}}$ of degree $n=m_{1}+m_{2}+\cdots+m_{s}$ by
\begin{equation}
V_{G}(P) = (V_{1} V_{2} \ldots V_{s})
\end{equation}
where $\alpha_{1},\alpha_{2},\ldots,\alpha_{s}$ are distinct elements of $\mathbb{F}$.\\
The block matrix $V_{k}$ is of order $n\times m_{k}, k=1,\ldots,s$ and defined to be the matrix
\begin{equation*}
V_{k}=V_{G}((x-\alpha_{k})^{m_{k}})
\end{equation*}
 with entry
\[  (V_{k})_{ij}= \left\{ \begin{array}{ll}
         \binom{i-1}{j-1}\alpha_{k}^{i-j} & \mbox{if $i \geq j$};\\
        0, & \mbox{otherwise }.\end{array} \right. \]
where $\binom{q}{p}$ is the binomial coefficients given by
\begin{align*}
\binom{q}{p}=\begin{cases}
\frac{q!}{p!(q-p)!} \quad\text{if}\quad p\leq q \\0 \quad\text{otherwise}
\end{cases}
\end{align*}
Throughout this article, $\mathbb{F}$ denotes a field of characteristic zero.\\
\indent In the following, we state a particular case of Vandermonde matrices
\begin{prop}
Let there be given the polynomials $L_{i}(x)=(x-\alpha)^{i},i=0,1,\ldots,n-1$, where $\alpha$ is an arbitrary element of $\mathbb{F}$,
The family of polynomials  $L_{i}(x),0\leq i\leq n-1$ is a basis of $\mathbb{F}_{n-1}[x]$.
The transpose of the change of basis matrix from this basis to the canonical basis is
the corresponding Vandermonde matrix
\begin{equation}
V_{G}((x-\alpha)^{n}) = \begin{pmatrix}
    1 &  0 & 0 &\cdots & 0   \\
   \alpha & 1 & 0 & \cdots &  0 \\
    \alpha^{2} &2\alpha &1 & \cdots & 0 \\
     \alpha^{3} &3\alpha^{2} &3\alpha & \cdots & 0 \\
     \alpha^{4} &4\alpha^{3} &6\alpha^{2} & \cdots & 0 \\
     \alpha^{5} &5\alpha^{4} &10\alpha^{3} & \cdots & 0 \\
    \vdots & \vdots &\vdots & \cdots & \vdots \\
    \alpha^{n-1} & (n-1)\alpha^{n-2} &\frac{(n-1)(n-2)}{2}\alpha^{n-3}   & \cdots & 1
\end{pmatrix}
\end{equation}
Then the corresponding inverse is
\begin{equation*}
 V_{G}^{-1}((x-\alpha)^{n}) = \begin{pmatrix}
    1 &  0 & 0 &\cdots & 0   \\
   -\alpha & 1 & 0 & \cdots &  0 \\
    \alpha^{2} &-2\alpha &1 & \cdots & 0 \\
     -\alpha^{3} &3\alpha^{2} &-3\alpha & \cdots & 0 \\
   \binom{4}{0}(-\alpha)^{4} & \binom{4}{1}(-\alpha)^{3} &\binom{4}{2}(-\alpha)^{2}   & \cdots & 0 \\
    \vdots & \vdots &\vdots & \cdots & \vdots \\
    \binom{n-1}{0}(-\alpha)^{n-1} & \binom{n-1}{1}(-\alpha)^{n-2} & \binom{n-1}{2}(-\alpha)^{n-3}   & \cdots & 1
\end{pmatrix}
\end{equation*}
\end{prop}
\begin{proof}
This result can be shown easily using the fact that
\begin{equation}
L_{i}(x)=\sum_{j=0}^{i}\binom{i}{j}(-\alpha)^{i-j}x^{j}
\end{equation}
for $i=0,1,\ldots,n-1$.
\end{proof}
 Now, we state a fundamental result which is to be proved
\begin{thm}
Let $P(x)=(x-\alpha_{1})^{m_{1}}(x-\alpha_{2})^{m_{2}}\cdots(x-\alpha_{s})^{m_{s}}$, where $\alpha_{1},\alpha_{2},\ldots,\alpha_{s}$ be differents elements of $\mathbb{F}$ and $m_{1},m_{2},\ldots,m_{s}$ be positive integers, such that $m_{1}+m_{2}+\cdots+m_{s}=n$. The family of polynomials
\begin{align}
B'=&\lbrace L_{10}(x)[P],L_{11}(x)[P],\ldots,L_{1m_{1}-1}(x)[P],\ldots,L_{s0}(x)[P],\nonumber
\\&L_{s1}(x)[P],\ldots,L_{sm_{s}-1}(x)[P] \rbrace
\end{align}
is a basis of $\mathbb{F}_{n-1}[x]$.
\end{thm}
\begin{proof}
Since $L_{jk_{j}}(x)[P]$ is of degree less than or equal to $n-1$ for all $1\leq j \leq s;0\leq k_{j} \leq m_{j}-1$, it suffices to verify that these polynomials are linearly independent. Let $u_{jk_{j}}\in \mathbb{F}$ such that
$$\sum_{i=0}^{m_{1}-1}u_{1i}L_{1i}(x)+\sum_{i=0}^{m_{2}-1}u_{2i}L_{2i}(x)+\cdots+\sum_{i=0}^{m_{s}-1}u_{si}L_{si}(x)=0$$
We differentiate this identity $l$ times, for $l\leq max(m_{1}-1,m_{2}-1,\ldots,m_{s}-1)$, and let $x=\alpha_{r}$, we obtain
 $$\sum_{i=0}^{m_{1}-1}u_{1i}l!\delta_{r1}\delta_{li}+\sum_{i=0}^{m_{2}-1}u_{2i}l!\delta_{r2}\delta_{li}+\cdots+\sum_{i=0}^{m_{s}-1}u_{si}l!\delta_{rs}\delta_{li}=0$$
 and then
 $$\sum_{i=0}^{m_{1}-1}u_{1i}\delta_{r1}\delta_{li}+\sum_{i=0}^{m_{2}-1}u_{2i}\delta_{r2}\delta_{li}+\cdots+\sum_{i=0}^{m_{s}-1}u_{si}\delta_{rs}\delta_{li}=0$$
 Thus, by a little manipulation of the last equality, the proof is completed.
\end{proof}
\begin{rem}~
\begin{itemize}
\item The coordinates in this basis of a polynomial $Q$ of degree less than or equal to $n-1$ are given explicitly by \eqref{eq: 2.5}.
\item The partial fractions decomposition is a powerful tool in many situations for example in the determination of the Jordan-Chevalley decomposition. Here we give a direct and simple method for obtaining the explicit formula for the partial fractions. Consider the polynomial $R(x)=(x-a_{1})^{m_{1}}(x-a_{2})^{m_{2}}\cdots(x-a_{s})^{m_{s}}$ of degree $n$ and $a_{1},a_{2},\ldots,a_{s}$ be distinct elements of $\mathbb{F}$. We have
\begin{equation*}
\sum_{j=1}^{s}L_{j0}(x)[R]=1
\end{equation*}
Consequently, we obtain
\begin{equation*}
\frac{1}{R(x)}=\sum_{j=1}^{s}\sum_{k_{j}=0}^{m_{j}-1}\frac{g^{(k_{j})}_{j}(a_{j})}{k_{j}!(x-a_{j})^{m_{j}-k_{j}}}
\end{equation*}
where $g_{j}(x)$ is defined by \eqref{eq: 2.4}.
\end{itemize}
\end{rem}
Let us consider the following linear confluent Vandermonde system
\begin{equation}\label{eq: 2.11}
\sum_{j=1}^{s}\sum_{k_{j}=0}^{m_{j}-1}x_{jk_{j}}\binom{r}{k_{j}}\alpha_{j}^{r-j}=u_{r}, 0\leq r\leq n-1
\end{equation}
The determination of the $x_{jk_{j}}$ is still related to the computation of the inverse of the confluent Vandermonde matrix. This topic was the subject of several papers. Therefore, various numerical and theoretical procedures have been developed for such as inversion. Recently, this issue has gained a regain of interest due to its appearance in applied mathematics, engineering computation, and many applications. In this section, we present a novel method to invert the confluent Vandermonde matrix and solve the system \eqref{eq: 2.11}. Our method is merely based on the interpolation polynomials. Consequently, an explicit expression for the inverse is obtained directly.\\
\indent Our main result is the following theorem which gives an explicit closed-form of the confluent Vandermonde inverse matrix.\\
\begin{thm}\label{Theorm 2.4}
Using the same notation above, the explicit inverse of the confluent Vandermonde matrix $V_{G}(P)$, where
\begin{equation*}
P(x)=(x-\alpha_{1})^{m_{1}}(x-\alpha_{2})^{m_{2}}\cdots(x-\alpha_{s})^{m_{s}}
\end{equation*}
has the form
\begin{equation}
V_{G}^{-1}(P) = \begin{pmatrix}
   \mathcal{L}_{1{m_{1}}} \\
   \mathcal{L}_{2{m_{2}}}  \\
    \vdots  \\
    \mathcal{L}_{sm_{s}}  \\

\end{pmatrix}
\end{equation}
for $r=1,2,\ldots,s$ the block matrix $\mathcal{L}_{rm_{r}}$ is of order $m_{r}\times n$ and given by
\begin{equation}
\mathcal{L}_{rm_{r}}=(\frac{1}{(j-1)!}L^{(j-1)}_{r(i-1)}(0)[P])_{1\leq i\leq m_{r},1\leq j\leq n}
\end{equation}
More precisely
\begin{equation*}
\mathcal{L}_{rm_{r}}= \begin{pmatrix}

  L_{r0}(0)[P] & L^{(1)}_{r0}(0)[P] & \cdots & \frac{1}{(n-1)!}L^{(n-1)}_{r0}(0)[P]  \\
  L_{r1}(0)[P] & L^{(1)}_{r1}(0)[P] & \cdots & \frac{1}{(n-1)!}L^{(n-1)}_{r1}(0)[P]  \\
    \vdots & \vdots & \cdots & \vdots \\
    L_{rm_{r}-1}(0)[P] & L^{(1)}_{rm_{r}-1}(0)[P] & \cdots & \frac{1}{(n-1)!}L^{(n-1)}_{rm_{r}-1}(0)[P]  \\
\end{pmatrix}
\end{equation*}
\end{thm}
\begin{proof}
 Let us write, every $Q\in \mathbb{F}_{n-1}[x]$ as
\begin{equation}
Q=\sum_{j=1}^{s}\sum_{k_{j}=0}^{m_{j}-1}\frac{1}{k_{j}!}Q^{(k_{j})}(\alpha_{j})L_{jk_{j}}(x)[P]
\end{equation}
In particular, for $Q=x^{r}, r=0,1,\ldots,n-1$, we have
\begin{equation}\label{eq: 2.15}
x^{r}=\sum_{j=1}^{s}\sum_{k_{j}=0}^{m_{j}-1}\frac{1}{k_{j}!}(x^{r})^{(k_{j})}(\alpha_{j})L_{jk_{j}}(x)[P]
\end{equation}
Thus, formula \eqref{eq: 2.15} can be representated in matrix form as
\begin{equation*}
\begin{pmatrix}
    V_{G}((x-\alpha_{1})^{m_{1}})&\cdots&V_{G}((x-\alpha_{s})^{m_{s}})  \\
\end{pmatrix}
\begin{pmatrix}
L_{10}(x)[P]\\
L_{11}(x)[P]\\
\vdots\\
L_{sm_{s}-1}(x)[P]
\end{pmatrix}
=
\begin{pmatrix}
1\\
x\\
\vdots\\
x^{n-1}
\end{pmatrix}
\end{equation*}
Using Taylor's Theorem, we get
\begin{equation}
L_{jk_{j}}(x)[P]=\sum_{i=0}^{n-1}\frac{L^{(i)}_{jk_{j}}(0)[P]}{i!}x^{i},
\end{equation}
for all $1\leq j \leq s$ and $0\leq k_{j} \leq m_{j}-1$.
From which we deduce
\begin{equation}
\begin{pmatrix}
L_{10}(x)[P]\\
L_{11}(x)[P]\\
\vdots\\
L_{sm_{s}-1}(x)[P]
\end{pmatrix}
=
\begin{pmatrix}
   \mathcal{L}_{1{m_{1}}} \\
   \mathcal{L}_{2{m_{2}}}  \\
    \vdots  \\
    \mathcal{L}_{sm_{s}}  \\

\end{pmatrix}
\begin{pmatrix}
1\\
x\\
\vdots\\
x^{n-1}
\end{pmatrix}
\end{equation}
Finally, we get the desired conclusion.
\end{proof}
As a consequence, we obtain
\begin{cor}
The solutions $x_{jk_{j}},j=1,2,\ldots,s$ of the linear confluent Vandermonde system \eqref{eq: 2.11} are given by
\begin{equation}
x_{jk_{j}}=\sum_{i=0}^{n-1}\frac{u_{i}}{i!}L^{(i)}_{jk_{j}}(0)[P]
\end{equation}
where the $L_{jk_{j}}(x)[P],1\leq j \leq k; 0\leq k_{j} \leq m_{j}-1,$ are described in \eqref{eq: 2.2}.
The solution $x_{jk_{j}}$ can be represented in matrix expression as follows
\begin{equation}
x_{jk_{j}}=\begin{pmatrix}
L_{jk_{j}}(0)[P] & \cdots & \frac{1}{(n-1)!}L^{(n-1)}_{jk_{j}}(0)[P]
\end{pmatrix}
\begin{pmatrix}
u_{0} \\
u_{1} \\
\vdots \\
u_{n-1}
\end{pmatrix}
\end{equation}
\end{cor}
\indent An interesting particular case is the following
\begin{cor}
Let $\alpha_{1},\alpha_{2}$ be two distinct elements of $\mathbb{F}$ and $m_{1},m_{2}$ be two positive integers. Let $V_{G}$ be the following confluent Vandermonde matrix
\begin{equation*}
V_{G} = \begin{pmatrix}
    1 &  0 &\cdots & 0 & 1 & 0 &\cdots & 0  \\
   \alpha_{1} & 1 & \cdots &  0 & \alpha_{2} & 1 & \cdots & 0 \\
    \alpha_{1}^{2} & 2\alpha_{1} & \cdots & 0& \alpha_{2}^{2} & 2\alpha_{2} & \cdots & 0  \\
     \alpha_{1}^{3} &3\alpha_{1}^{2} & \cdots & 0 & \alpha_{2}^{3} & 3\alpha_{2}^{2} & \cdots & 0   \\
     \alpha_{1}^{4} &4\alpha_{1}^{3} & \cdots & 0 & \alpha_{2}^{4} & 4\alpha_{2}^{3} & \cdots & 0  \\
     \alpha_{1}^{5} &5\alpha_{1}^{4} & \cdots & 0 & \alpha_{2}^{5} & 5\alpha_{2}^{4} & \cdots & 0  \\
    \vdots & \vdots & \cdots & \vdots & \vdots & \vdots & \cdots & \vdots \\
    \alpha_{1}^{n-1} & (n-1)\alpha_{1}^{n-2}   & \cdots & 1& \alpha_{2}^{n-1} & (n-1)\alpha_{2}^{n-2} &\cdots&1
\end{pmatrix}
\end{equation*}
where $m_{1}+m_{2}=n$. Then the inverse of $V_{G}$ is
\begin{equation*}
V_{G}^{-1}= \begin{pmatrix}
  a_{11} & a_{12} & a_{13} &\cdots & a_{1n}  \\
  a_{21} & a_{22} & a_{23} &\cdots & a_{2n}  \\
    \vdots & \vdots & \vdots & \cdots & \vdots \\
    a_{m_{1}1} & a_{m_{1}2} & a_{m_{1}3} & \cdots & a_{m_{1}n}  \\
   b_{11} & b_{12} & b_{13} & \cdots & b_{1n}  \\
  b_{21} & b_{22} & b_{23} & \cdots & b_{2n}  \\
    \vdots & \vdots & \vdots & \cdots & \vdots \\
    b_{m_{2}1} & b_{m_{2}2} & b_{m_{2}3} & \cdots & b_{m_{2}n}  \\
\end{pmatrix}
\end{equation*}
where
\begin{equation*}
a_{ij}=\sum_{p=0}^{m_{1}-i}\frac{(-1)^{m_{2}+i+j}\binom{m_{2}+p-1}{m_{2}-1}}{(\alpha_{1}-\alpha_{2})^{m_{2}+p}}\sum_{r=0}^{j-1}\binom{i+p-1}{r}\binom{m_{2}}{j-1-r}\alpha_{1}^{p+i-1-r}\alpha_{2}^{m_{2}-j+1+r}
\end{equation*}
and
\begin{equation*}
b_{ij}=\sum_{p=0}^{m_{2}-i}\frac{(-1)^{m_{1}+i+j}\binom{m_{1}+p-1}{m_{1}-1}}{(\alpha_{2}-\alpha_{1})^{m_{1}+p}}\sum_{r=0}^{j-1}\binom{m_{1}}{j-1-r}\binom{i+p-1}{r}\alpha_{1}^{m_{1}-j+1+r}\alpha_{2}^{p+i-1-r}
\end{equation*}
for $1\leq j\leq n$ and $1\leq i\leq m_{p}$ with $p=1,2$.
\end{cor}
\begin{proof}
Let $P(x)=(x-\alpha_{1})^{m_{1}}(x-\alpha_{2})^{m_{2}}$. Then we can see easily that for all $0\leq k_{1}\leq m_{1}-1$ and $0\leq k_{2}\leq m_{2}-1$, we have
$$ \left \{
\begin{array}{rcl}
L_{1k_{1}}(x)[P]&=&\sum_{i=0}^{m_{1}-k_{1}-1}\frac{(-1)^{i}\binom{m_{2}+i-1}{m_{2}-1}}{(\alpha_{1}-\alpha_{2})^{m_{2}+i}}(x-\alpha_{1})^{i+k_{1}}(x-\alpha_{2})^{m_{2}}\\
L_{2k_{2}}(x)[P]&=&\sum_{i=0}^{m_{2}-k_{2}-1}\frac{(-1)^{i}\binom{m_{1}+i-1}{m_{1}-1}}{(\alpha_{2}-\alpha_{1})^{m_{1}+i}}(x-\alpha_{1})^{m_{1}}(x-\alpha_{2})^{i+k_{2}}
\end{array}
\right.
$$
As a consequence of Theorem~\ref{Theorm 2.4}, we get the result using only Leibniz's rule.
\end{proof}
With the following example, we illustrate the use of Theorem~\ref{Theorm 2.4}
\begin{example}\label{example 1}
Let $\alpha_{1},\alpha_{2}$ distinct elements of $\mathbb{F}$. We wish to inverse the matrix
$$ K = \begin{pmatrix}
    1 &  0 & 0 & 1   \\
   \alpha_{1} & 1 & 0 & \alpha_{2} \\
    \alpha^{2}_{1} &2\alpha_{1} &1 & \alpha_{2}^{2} \\
     \alpha^{3}_{1} &3\alpha^{2}_{1} &3\alpha_{1} & \alpha_{2}^{3} \\
\end{pmatrix}
$$
In this case, we consider
$$P_{1}(x)=(x-\alpha_{2}),g_{1}(x)=(P_{1}(x))^{-1} $$
$$P_{2}(x)=(x-\alpha_{1})^{3},g_{2}(x)=(P_{2}(x))^{-1}$$
On the other hand
$$ L_{1k_{1}}(x)[P]=P_{1}(x)(x-\alpha_{1})^{k_{1}}\sum_{i=0}^{2-k_{1}}\frac{1}{i!} g^{(i)}_{1}(\alpha_{1})(x-\alpha_{1})^{i},0\leq k_{1}\leq 2$$
and
$$ L_{20}(x)[P]=P_{2}(x)g_{2}(\alpha_{2})$$
Make on the polynomials $L_{jk_{j}}(x)[P]$ elementary operations and expressing them in the canonical basis of $\mathbb{F}_{3}[x]$, we then obtain
$$ \left \{
\begin{array}{rcl}
L_{10}(x)[P]=&\frac{-3\alpha_{1}^{2}\alpha_{2}+3\alpha_{1}\alpha_{2}^{2}-\alpha_{2}^{3}}{(\alpha_{1}-\alpha_{2})^{3}}+\frac{3\alpha_{1}^{2}}{(\alpha_{1}-\alpha_{2})^{3}}x+\frac{-3\alpha_{1}}{(\alpha_{1}-\alpha_{2})^{3}}x^{2}+\frac{1}{(\alpha_{1}-\alpha_{2})^{3}}x^{3}\\
L_{11}(x)[P]=&\frac{2\alpha_{1}^{2}\alpha_{2}-\alpha_{1}\alpha_{2}^{2}}{(\alpha_{1}-\alpha_{2})^{2}}+\frac{\alpha_{2}^{2}-2\alpha_{1}^{2}-2\alpha_{1}\alpha_{2}}{(\alpha_{1}-\alpha_{2})^{2}}x+\frac{3\alpha_{1}}{(\alpha_{1}-\alpha_{2})^{2}}x^{2}+\frac{-1}{(\alpha_{1}-\alpha_{2})^{2}}x^{3}\\
L_{12}(x)[P]=&\frac{-\alpha_{1}^{2}\alpha_{2}}{(\alpha_{1}-\alpha_{2})}+\frac{\alpha_{1}^{2}+2\alpha_{1}\alpha_{2}}{(\alpha_{1}-\alpha_{2})}x-\frac{2\alpha_{1}+\alpha_{2}}{(\alpha_{1}-\alpha_{2})}x^{2}+\frac{1}{(\alpha_{1}-\alpha_{2})}x^{3}\\
L_{20}(x)[P]=&\frac{-\alpha^{3}_{1}}{(\alpha_{2}-\alpha_{1})^{3}}+\frac{3\alpha_{1}^{2}}{(\alpha_{2}-\alpha_{1})^{3}}x-\frac{3\alpha_{1}}{(\alpha_{2}-\alpha_{1})^{3}}x^{2}+\frac{1}{(\alpha_{2}-\alpha_{1})^{3}}x^{3}
\end{array}
\right.
$$
Using Theorem~\ref{Theorm 2.4}, we get
$$ K^{-1} = \begin{pmatrix}

  \frac{-3\alpha_{1}^{2}\alpha_{2}+3\alpha_{1}\alpha_{2}^{2}-\alpha_{2}^{3}}{(\alpha_{1}-\alpha_{2})^{3}}&\frac{3\alpha_{1}^{2}}{(\alpha_{1}-\alpha_{2})^{3}}&\frac{-3\alpha_{1}}{(\alpha_{1}-\alpha_{2})^{3}} &\frac{1}{(\alpha_{1}-\alpha_{2})^{3}} \\
    \frac{2\alpha_{1}^{2}\alpha_{2}-\alpha_{1}\alpha_{2}^{2}}{(\alpha_{1}-\alpha_{2})^{2}} &\frac{\alpha_{2}^{2}-2\alpha_{1}^{2}-2\alpha_{1}\alpha_{2}}{(\alpha_{1}-\alpha_{2})^{2}} &\frac{3\alpha_{1}}{(\alpha_{1}-\alpha_{2})^{2}} & \frac{-1}{(\alpha_{1}-\alpha_{2})^{2}} \\
    \frac{-\alpha_{1}^{2}\alpha_{2}}{(\alpha_{1}-\alpha_{2})} &\frac{\alpha_{1}^{2}+2\alpha_{1}\alpha_{2}}{(\alpha_{1}-\alpha_{2})} &\frac{-2\alpha_{1}-\alpha_{2}}{(\alpha_{1}-\alpha_{2})} & \frac{1}{(\alpha_{1}-\alpha_{2})}\\
    \frac{-\alpha^{3}_{1}}{(\alpha_{2}-\alpha_{1})^{3}} &\frac{3\alpha^{2}_{1}}{(\alpha_{2}-\alpha_{1})^{3}} &\frac{-3\alpha_{1}}{(\alpha_{2}-\alpha_{1})^{3}} & \frac{1}{(\alpha_{2}-\alpha_{1})^{3}}\\
\end{pmatrix}
$$
\end{example}
\begin{rem}
A companion matrix is an essential tool in the area of linear algebra. Companion matrice is extensively employed in control theory, more precisely in the observable canonical form as well as the controllable canonical form. It plays a crucial role in linear recursive sequences and linear differential equations. The usual and confluent Vandermonde matrices are deeply connected with the companion matrix.
If we write the polynomial $P$ defined in \eqref{eq: 2.1} as
\begin{equation}
P(x)=x^{n}-\sum_{i=1}^{n}a_{i}x^{n-i}
\end{equation}
Then the companion matrix of the polynomial $P$ noted $C_{P}$ is defined to be
\begin{equation}\label{eq: 2.22}
C_{P}=\begin{pmatrix}
    0 &  1 & 0 & \cdots & 0 & 0 \\
    0 & 0 & 1 & \cdots & 0 & 0\\
    \vdots & \vdots & \vdots &  & \vdots & \vdots\\
      0 & 0 & 0 & \cdots & 0 & 1 \\
    a_{n} & a_{n-1}& a_{n-2} & \cdots & a_{2} & a_{1} \\
\end{pmatrix}
\end{equation}
It may be noted that the confluent Vandermonde matrix is the change of basis matrix which makes the companion matrix into Jordan canonical form J:
\begin{equation}
C_{P}=V_{G}(P)JV_{G}^{-1}(P)
\end{equation}
\end{rem}
Now, we are concerned with a more general case
\begin{defn}
Let $\alpha_{1},\alpha_{2},\ldots,\alpha_{s}$ be non-zero distinct pairwise elements of $\mathbb{F}$ and $r_{i}=(r_{i1},r_{i2}\ldots,r_{im_{i}}), i=1,\ldots, s$ where the $r_{ij}$ are nonnegative integers and let us denote $\mathbf{r_{s}}=(r_{1},r_{2},\ldots,r_{s})$.
The $\mathbf{r_{s}}$-confluent Vandermonde matrix related to $\alpha_{1},\alpha_{2},\ldots,\alpha_{s}$ is defined to be the following matrix
\begin{equation*}
V^{\mathbf{r_{s}}}_{G}(P) = \begin{pmatrix}
   V^{r_{1}}_{G}((x-\alpha_{1})^{m_{1}}) &  V^{r_{2}}_{G}((x-\alpha_{2})^{m_{2}}) & \cdots &  V^{r_{s}}_{G}((x-\alpha_{s})^{m_{s}})\\
\end{pmatrix}
\end{equation*}
For $k=1,2,\ldots,s$ the block matrix $(V^{r_{k}}_{G}((x-\alpha_{k})^{m_{k}})$ is of order $n\times m_{k}$ with entry
\begin{equation}
  (V^{r_{k}}_{G}((x-\alpha_{k})^{m_{k}}))_{ij}= \left\{ \begin{array}{ll}
         \binom{i-1}{j-1}\alpha_{k}^{r_{kj}+i-j} & \mbox{if $i \geq j$};\\
        0, & \mbox{otherwise }.\end{array} \right.
\end{equation}
\end{defn}
For any $\mathbf{r_{s}}$, the inverse of the $\mathbf{r_{s}}$-confluent Vandermonde matrix related to $\alpha_{1},\alpha_{2},\ldots,\alpha_{s}$ is obtained easily once the inverse of the confluent Vandermonde is available.
\begin{example}
We now apply our method to derive the inverse of the following $\mathbf{r_{2}}$-confluent Vandermonde matrix
$$ K = \begin{pmatrix}
    \alpha_{1}^{r_{11}} & 0 & 0 & \alpha_{2}^{r_{21}}   \\
   \alpha_{1}^{r_{11}+1} & \alpha_{1}^{r_{12}} & 0 & \alpha_{2}^{r_{21}+1} \\
    \alpha_{1}^{r_{11}+2} &2\alpha_{1}^{r_{12}+1} &\alpha_{1}^{r_{13}} & \alpha_{2}^{r_{21}+2} \\
    \alpha_{1}^{r_{11}+3} &3\alpha_{1}^{r_{12}+2} &3\alpha_{1}^{r_{13}+1} & \alpha_{2}^{r_{21}+3} \\
\end{pmatrix}
$$
where $\alpha_{1},\alpha_{2}$ are non-zero distinct elements of $\mathbb{F}$. Using the Example~\ref{example 1}, we obtain
$$ K^{-1} = \begin{pmatrix}

  \frac{-3\alpha_{1}^{2}\alpha_{2}+3\alpha_{1}\alpha_{2}^{2}-\alpha_{2}^{3}}{\alpha_{1}^{r_{11}}(\alpha_{1}-\alpha_{2})^{3}}&\frac{3\alpha_{1}^{2}}{\alpha_{1}^{r_{11}}(\alpha_{1}-\alpha_{2})^{3}}&\frac{-3\alpha_{1}}{\alpha_{1}^{r_{11}}(\alpha_{1}-\alpha_{2})^{3}} &\frac{1}{\alpha_{1}^{r_{11}}(\alpha_{1}-\alpha_{2})^{3}} \\
    \frac{2\alpha_{1}^{2}\alpha_{2}-\alpha_{1}\alpha_{2}^{2}}{\alpha_{1}^{r_{12}}(\alpha_{1}-\alpha_{2})^{2}} &\frac{\alpha_{2}^{2}-2\alpha_{1}^{2}-2\alpha_{1}\alpha_{2}}{\alpha_{1}^{r_{12}}(\alpha_{1}-\alpha_{2})^{2}} &\frac{3\alpha_{1}}{\alpha_{1}^{r_{12}}(\alpha_{1}-\alpha_{2})^{2}} & \frac{-1}{\alpha_{1}^{r_{12}}(\alpha_{1}-\alpha_{2})^{2}} \\
    \frac{-\alpha_{1}^{2}\alpha_{2}}{\alpha_{1}^{r_{13}}(\alpha_{1}-\alpha_{2})} &\frac{\alpha_{1}^{2}+2\alpha_{1}\alpha_{2}}{\alpha_{1}^{r_{13}}(\alpha_{1}-\alpha_{2})} &\frac{-2\alpha_{1}-\alpha_{2}}{\alpha_{1}^{r_{13}}(\alpha_{1}-\alpha_{2})} & \frac{1}{\alpha_{1}^{r_{13}}(\alpha_{1}-\alpha_{2})}\\
    \frac{-\alpha^{3}_{1}}{\alpha_{2}^{r_{21}}(\alpha_{2}-\alpha_{1})^{3}} &\frac{3\alpha^{2}_{1}}{\alpha_{2}^{r_{21}}(\alpha_{2}-\alpha_{1})^{3}} &\frac{-3\alpha_{1}}{\alpha_{2}^{r_{21}}(\alpha_{2}-\alpha_{1})^{3}} & \frac{1}{\alpha_{2}^{r_{21}}(\alpha_{2}-\alpha_{1})^{3}}\\
\end{pmatrix}
$$
\end{example}
\section{Explicit inverse of usual Vandermonde matrices and application}
In linear algebra, nonsingular linear Vandermonde systems appear naturally in several situations. Various methods have been presented to determine explicit solutions of the nonsingular linear Vandermonde systems see, e.g. \cite{Apant}. We are interested to present a new method for determining solutions to this type of linear system. Let us fix $k$ element $u_{0},u_{1},\ldots,u_{k-1}$ of $\mathbb{K}$, and consider the following nonsingular linear Vandermonde system,
\begin{equation}\label{eq: 2.24}
\sum_{i=1}^{k}\alpha_{i}^{n}x_{i}=u_{n}, n=0,1,\ldots,k-1
\end{equation}
where $\alpha_{i}, i=1,2,\ldots,k$ are pairwise distinct elements of $\mathbb{K}$. The most elegant method for solving the system \eqref{eq: 2.24} is based on the determination of the inverse of the usual Vandermonde matrix $V$ defined by
\begin{equation}\label{eq: 2.25}
V = \begin{pmatrix}
    1 &  1 & \cdots & 1   \\
   \alpha_{1} & \alpha_{2} & \cdots &  \alpha_{k} \\
    \alpha_{1}^{2} & \alpha_{2}^{2} & \cdots & \alpha_{k}^{2} \\
    \vdots & \vdots & \cdots & \vdots \\
    \alpha_{1}^{k-2} & \alpha_{2}^{k-2} & \cdots & \alpha_{k}^{k-2} \\
    \alpha_{1}^{k-1} &  \alpha_{2}^{k-1}   & \cdots & \alpha_{k}^{k-1}
\end{pmatrix}
\end{equation}
There are two famous methods of calculating the explicit inverse of the usual Vandermonde matrices. The first one is the $LU$ factorization it aims to provide an explicit formula of the inverse with a compact expression but it is necessary to compute the matrices $L$ and $U$, their inverses, and the matrix product $U^{-1}L^{-1}$ (see \cite{Hcli,Lrtur,Slyang}). The second inversion method is based on the use of the elementary symmetric functions of roots defined by the following expressions
\begin{align}
\left\{
\begin{array}{cccccccccc}
\sigma_{0}(\alpha_{1},\alpha_{2},\ldots,\alpha_{k})& = & 1  \\
\sigma_{1}(\alpha_{1},\alpha_{2},\ldots,\alpha_{k}) & = & \alpha_{1}+\alpha_{2}+\cdots+\alpha_{k} \\
\vdots & \vdots & \vdots & \\
\sigma_{r}(\alpha_{1},\alpha_{2},\ldots,\alpha_{k}) & = & \sum_{\substack{1\leq j_{1}< \cdots< j_{r}\leq k\phantom{-}}}\alpha_{j_{1}}\alpha_{j_{2}}\cdots\alpha_{j_{r}} \\
\vdots & \vdots & \vdots & \\
\sigma_{k}(\alpha_{1},\alpha_{2},\ldots,\alpha_{k}) & = & \alpha_{1}\alpha_{2}\cdots\alpha_{k} \\
\end{array}
\right.
\end{align}
Let us denote
\begin{equation}
\sigma_{r}(j)=\sigma_{r}(\alpha_{1},\ldots,\alpha_{j-1},\alpha_{j+1},\ldots,\alpha_{k})
\end{equation}
where $r=1,2,\ldots,k$. It has been demonstrated in \cite{Meael} that the explicit inverse of the usual Vandermonde matrices, using the elementary symmetric functions of roots, is
\begin{equation}
V^{-1}=(v_{ij})_{\substack{1\leq i,j \leq k\phantom{-}}}
\end{equation}
where the $v_{ij}$'s are given by
\begin{equation}
v_{ij}=\frac{(-1)^{k+j}\sigma_{k-j}(i)}{\prod_{p=1,p \neq i}^{k}(\alpha_{i}-\alpha_{p})}
\end{equation}
With our main result, we can easily  obtain the inverse of the matrix $V$ without utilizing the elementary symmetric function of root and the $LU$ factorization. When $m_{1}=m_{2}=\cdots=m_{k}=1$ it suffices to take
\begin{equation}
L_{j0}(x)[P]=g_{j}(\alpha_j)P_{j}(x),1\leq j \leq k.
\end{equation}
which coincides with the polynomials of Lagrange. That is, this polynomials are defined by
\begin{equation}
L_{j0}(x)[P]=\prod_{i=1,i\neq j}^{k}\frac{x-\alpha_{i}}{\alpha_{j}-\alpha_{i}}
\end{equation}
for $1\leq j \leq k$. More precisely, we have the following corollary
\begin{cor}\label{cor 3.1}
Let $\alpha_{1},\alpha_{2},\ldots,\alpha_{k}$ distinct elements of $\mathbb{F}$.
The inverse of the usual Vandermonde matrix
\begin{equation}
V = \begin{pmatrix}
    1 &  1 & \cdots & 1   \\
   \alpha_{1} & \alpha_{2} & \cdots &  \alpha_{k} \\
    \alpha_{1}^{2} & \alpha_{2}^{2} & \cdots & \alpha_{k}^{2} \\
    \vdots & \vdots & \cdots & \vdots \\
    \alpha_{1}^{k-2} & \alpha_{2}^{k-2} & \cdots & \alpha_{k}^{k-2} \\
    \alpha_{1}^{k-1} &  \alpha_{2}^{k-1}   & \cdots & \alpha_{k}^{k-1}
\end{pmatrix}
\end{equation}
is given by
\begin{equation}
V^{-1} = \begin{pmatrix}
    s_{1}(0) &  s^{(1)}_{1}(0) & \frac{1}{2!}s^{(2)}_{1}(0) &\cdots &   \frac{1}{(k-1)!}s^{(k-1)}_{1}(0) \\
    s_{2}(0) &  s^{(1)}_{2}(0) & \frac{1}{2!}s^{(2)}_{2}(0) &\cdots &   \frac{1}{(k-1)!}s^{(k-1)}_{2}(0) \\
    s_{3}(0) &  s^{(1)}_{3}(0) & \frac{1}{2!}s^{(2)}_{3}(0) &\cdots &   \frac{1}{(k-1)!}s^{(k-1)}_{3}(0) \\
    \vdots & \vdots & \vdots & \cdots & \vdots  \\
    s_{k}(0) &  s^{(1)}_{k}(0) & \frac{1}{2!}s^{(2)}_{k}(0) &\cdots &   \frac{1}{(k-1)!}s^{(k-1)}_{k}(0) \\
\end{pmatrix}
\end{equation}
where
\begin{equation*}
s_{j}(x)=\prod_{i=1,i\neq j}^{k}\frac{x-\alpha_{i}}{\alpha_{j}-\alpha_{i}}
\end{equation*}
\end{cor}
\begin{proof}
It is a particular case of the main result.
\end{proof}
\begin{rem}In the case of Corollary~\ref{cor 3.1}, the companion matrix $C_{P}$ defined by \eqref{eq: 2.22} is diagonalizable. It is a well-known fact that the characetristic polynomial of $C_{P}$ is $P(x)=\prod_{j=1}^{k}(x-\alpha_{j})$.
Let $v_{i}$ the vector colomn defined by
\begin{equation*}
v_{i}=(1,\alpha_{i},\alpha_{i}^{2},\ldots,\alpha_{i}^{k-1})^{T}, 1\leq i\leq k
\end{equation*}
for $i=1,2,\ldots,k$. It is easy to check
\begin{equation*}
C_{P}v_{i}=\alpha_{i}v_{i}, 1\leq i\leq k
\end{equation*}
which means that $v_{i}, 1\leq i\leq k$, are eigenvectors of $C_{P}$. Consequently, we have
\begin{equation*}
C_{P}=VDV^{-1}
\end{equation*}
where $D$ is the diagonal matrix $D=diag(\alpha_{1},\alpha_{2},\ldots,\alpha_{k})$ and $V$ is the usual Vandermonde matrix defined by \eqref{eq: 2.25}.
\end{rem}
\begin{prop}
The solutions $x_{i},i=1,2,\ldots,k$ of the nonsingular linear Vandermonde system \eqref{eq: 2.24} are given by
\begin{equation}\label{eq: 2.34}
x_{i}=\sum_{j=0}^{k-1}\frac{u_{j}}{j!}s^{(j)}_{i}(0)
\end{equation}
where
\begin{equation*}
s_{j}(x)=\prod_{i=1,i\neq j}^{k}\frac{x-\alpha_{i}}{\alpha_{j}-\alpha_{i}}
\end{equation*}
The solution $x_{i}$ can be represented in matrix expression as follows
\begin{equation}
x_{i}=\begin{pmatrix}
s_{i}(0) & \frac{1}{1!}s^{(1)}_{i}(0) & \cdots & \frac{1}{(k-1)!}s^{(k-1)}_{i}(0)
\end{pmatrix}
\begin{pmatrix}
u_{0} \\
u_{1} \\
\vdots \\
u_{k-1}
\end{pmatrix}
\end{equation}
\end{prop}
\begin{rem}~
\begin{itemize}
\item The Vandermonde systems represent a fundamental tool in the control and system theory see e.g.\cite{Apant}. The usual Vandermonde matrix has received much interest in cryptography, it is a helpful tool in decoding the Reed-Solomon codes see, e.g. \cite{Sireed}.
\item In the large literature, solutions of the linear system of Vandermonde \eqref{eq: 2.24} are not known under the form \eqref{eq: 2.34}.
\end{itemize}
\end{rem}
\begin{example}
For giving illustration, let $\alpha_{1},\alpha_{2},\alpha_{3},\alpha_{4}$ distinct elements of $\mathbb{F}$. We wish to inverse the matrix
$$ K = \begin{pmatrix}
    1 &  1 & 1 & 1   \\
   \alpha_{1} & \alpha_{2} & \alpha_{3} & \alpha_{4} \\
    \alpha^{2}_{1} &\alpha_{2}^{2} &\alpha^{2}_{3} & \alpha_{4}^{2} \\
     \alpha^{3}_{1} &\alpha^{3}_{2} &\alpha_{3}^{3} & \alpha_{4}^{3} \\
\end{pmatrix}
$$
we consider
$$ \left \{
\begin{array}{rcl}
s_{1}(x)[P]&=&\frac{(x-\alpha_{2})(x-\alpha_{3})(x-\alpha_{4})}{(\alpha_{1}-\alpha_{2})(\alpha_{1}-\alpha_{3})(\alpha_{1}-\alpha_{4})}\\
s_{2}(x)[P]&=&\frac{(x-\alpha_{1})(x-\alpha_{3})(x-\alpha_{4})}{(\alpha_{2}-\alpha_{1})(\alpha_{2}-\alpha_{3})(\alpha_{2}-\alpha_{4})}\\
s_{3}(x)[P]&=&\frac{(x-\alpha_{1})(x-\alpha_{2})(x-\alpha_{4})}{(\alpha_{3}-\alpha_{1})(\alpha_{3}-\alpha_{2})(\alpha_{3}-\alpha_{4})}\\
s_{4}(x)[P]&=&\frac{(x-\alpha_{1})(x-\alpha_{2})(x-\alpha_{3})}{(\alpha_{4}-\alpha_{1})(\alpha_{4}-\alpha_{2})(\alpha_{4}-\alpha_{3})}
\end{array}
\right.
$$
Expressing these polynomials in the canonical basis of $\mathbb{F}_{3}[x]$, we obtain
$$ \left \{
\begin{array}{rcl}
s_{1}(x)[P]&=&\frac{-(\alpha_{2}\alpha_{3}\alpha_{4})+(\alpha_{2}\alpha_{3}+\alpha_{2}\alpha_{4}+\alpha_{3}\alpha_{4})x-(\alpha_{2}+\alpha_{3}+\alpha_{4})x^{2}+x^{3}}{(\alpha_{1}-\alpha_{2})(\alpha_{1}-\alpha_{3})(\alpha_{1}-\alpha_{4})}\\
s_{2}(x)[P]&=&\frac{-(\alpha_{1}\alpha_{3}\alpha_{4})+(\alpha_{1}\alpha_{3}+\alpha_{1}\alpha_{4}+\alpha_{3}\alpha_{4})x-(\alpha_{1}+\alpha_{3}+\alpha_{4})x^{2}+x^{3}}{(\alpha_{2}-\alpha_{1})(\alpha_{2}-\alpha_{3})(\alpha_{2}-\alpha_{4})}\\
s_{3}(x)[P]&=&\frac{-(\alpha_{1}\alpha_{2}\alpha_{4})+(\alpha_{1}\alpha_{2}+\alpha_{1}\alpha_{4}+\alpha_{2}\alpha_{4})x-(\alpha_{1}+\alpha_{2}+\alpha_{4})x^{2}+x^{3}}{(\alpha_{3}-\alpha_{1})(\alpha_{3}-\alpha_{2})(\alpha_{3}-\alpha_{4})}\\
s_{4}(x)[P]&=&\frac{-(\alpha_{1}\alpha_{2}\alpha_{3})+(\alpha_{1}\alpha_{2}+\alpha_{1}\alpha_{3}+\alpha_{2}\alpha_{3})x-(\alpha_{1}+\alpha_{2}+\alpha_{4})x^{2}+x^{3}}{(\alpha_{4}-\alpha_{1})(\alpha_{4}-\alpha_{2})(\alpha_{4}-\alpha_{3})}
\end{array}
\right.
$$
Using Corollary~\ref{cor 3.1}, the inverse of the matrix $K$ is
$$\begin{pmatrix}
\frac{-(\alpha_{2}\alpha_{3}\alpha_{4})}{(\alpha_{1}-\alpha_{2})(\alpha_{1}-\alpha_{3})(\alpha_{1}-\alpha_{4})} &\frac{(\alpha_{2}\alpha_{3}+\alpha_{2}\alpha_{4}+\alpha_{3}\alpha_{4})}{(\alpha_{1}-\alpha_{2})(\alpha_{1}-\alpha_{3})(\alpha_{1}-\alpha_{4})} &\frac{-(\alpha_{2}+\alpha_{3}+\alpha_{4})}{(\alpha_{1}-\alpha_{2})(\alpha_{1}-\alpha_{3})(\alpha_{1}-\alpha_{4})} & \frac{1}{(\alpha_{1}-\alpha_{2})(\alpha_{1}-\alpha_{3})(\alpha_{1}-\alpha_{4})} \\
    \frac{-(\alpha_{1}\alpha_{3}\alpha_{4})}{(\alpha_{2}-\alpha_{1})(\alpha_{2}-\alpha_{3})(\alpha_{2}-\alpha_{4})} &\frac{(\alpha_{1}\alpha_{3}+\alpha_{1}\alpha_{4}+\alpha_{3}\alpha_{4})}{(\alpha_{2}-\alpha_{1})(\alpha_{2}-\alpha_{3})(\alpha_{2}-\alpha_{4})} &\frac{-(\alpha_{1}+\alpha_{3}+\alpha_{4})}{(\alpha_{2}-\alpha_{1})(\alpha_{2}-\alpha_{3})(\alpha_{2}-\alpha_{4})} & \frac{1}{(\alpha_{2}-\alpha_{1})(\alpha_{2}-\alpha_{3})(\alpha_{2}-\alpha_{4})} \\
 \frac{-(\alpha_{1}\alpha_{2}\alpha_{4})}{(\alpha_{3}-\alpha_{1})(\alpha_{3}-\alpha_{2})(\alpha_{3}-\alpha_{4})} &\frac{(\alpha_{1}\alpha_{2}+\alpha_{1}\alpha_{4}+\alpha_{2}\alpha_{4})}{(\alpha_{3}-\alpha_{1})(\alpha_{3}-\alpha_{2})(\alpha_{3}-\alpha_{4})} &\frac{-(\alpha_{1}+\alpha_{2}+\alpha_{4})}{(\alpha_{3}-\alpha_{1})(\alpha_{3}-\alpha_{2})(\alpha_{3}-\alpha_{4})} & \frac{1}{(\alpha_{3}-\alpha_{1})(\alpha_{3}-\alpha_{2})(\alpha_{3}-\alpha_{4})} \\
 \frac{-(\alpha_{1}\alpha_{2}\alpha_{3})}{(\alpha_{4}-\alpha_{1})(\alpha_{4}-\alpha_{2})(\alpha_{4}-\alpha_{3})} &\frac{(\alpha_{1}\alpha_{2}+\alpha_{1}\alpha_{3}+\alpha_{2}\alpha_{3})}{(\alpha_{4}-\alpha_{1})(\alpha_{4}-\alpha_{2})(\alpha_{4}-\alpha_{3})} &\frac{-(\alpha_{1}+\alpha_{2}+\alpha_{3})}{(\alpha_{4}-\alpha_{1})(\alpha_{4}-\alpha_{2})(\alpha_{4}-\alpha_{3})} & \frac{1}{(\alpha_{4}-\alpha_{1})(\alpha_{4}-\alpha_{2})(\alpha_{4}-\alpha_{3})}
\end{pmatrix}$$
\end{example}
\section{Conclusion}
 The important role of the generalization of Hermite's interpolation for the confluent Vandermondes matrices has shown, in providing a new method for inverting this kind of matrices. It is well known that obtaining the inverse of the confluent Vandermonde matrices is fairly difficult but with our method, we have made the problem to a calculation of Taylor's expansion.\\
Furthermore, the comparison with the existing methods and algorithms leads us to say that our method is direct and is manageable for inverting the confluent Vandermonde matrices. The advantage of our method is that is direct and straightforward, it gives explicit and compact formulas not exist in the large literature on this essential topic.\\
Certainly, Our results may have many intersting perspectives in diverse areas of mathematics and natural sciences, notably on the situations where the Vandermonde matrices have acquired much usefulness.\\

\end{document}